\documentclass[a4paper]{amsproc}

\usepackage[english]{babel}
\usepackage{tikz}
\usetikzlibrary{positioning,arrows}

\usepackage{amsmath,amssymb,amsthm}

\usepackage{enumerate}

\newtheorem{theorem}{Theorem}[section]
\newtheorem{proposition}[theorem]{Proposition}
\newtheorem{lemma}[theorem]{Lemma}
\newtheorem{corollary}[theorem]{Corollary}

\newtheorem{conjecture}[theorem]{Conjecture}

\theoremstyle{definition}
\newtheorem{definition}[theorem]{Definition}

\theoremstyle{remark}
\newtheorem{example}[theorem]{Example}

\renewcommand{\Im}{\operatorname{Im}}
\newcommand{\id}{\operatorname{id}}

\newcommand{\Aut}{\operatorname{Aut}}

\newcommand{\coker}{\operatorname{coker}}
\newcommand{\rank}{\operatorname{rank}}
\newcommand{\GL}{\operatorname{GL}}
\newcommand{\ord}{\operatorname{ord}}

\newcommand{\Z}{{\mathbb Z}}

\newcommand{\E}{{\mathcal E}}

\newcommand{\B}{{\mathcal B}}

\usepackage[pdfpagelabels]{hyperref}

\newcommand{\xrightarrowdbl}[2][]{%
  \xrightarrow[#1]{#2}\mathrel{\mkern-14mu}\rightarrow
}

\AtBeginDocument{%
   \def\MR#1{}
}

\begin{document}

\title{The group of self-homotopy equivalences of $A_n^2$-polyhedra}
\author{Cristina Costoya}
\address{CITIC, Departamento de Computaci\'on,
Universidade da Coru{\~n}a, 15071-A Coru{\~n}a, Spain.}
\email{cristina.costoya@udc.es}
\thanks{The first author was partially supported by Mi\-nis\-te\-rio de Econom\'ia y Competitividad (Spain), grant  MTM2016-79661-P}
\author{David M\'endez}
\address{Departamento de \'Algebra, Geometr{\'\i}a y Topolog{\'\i}a, Universidad de M{\'a}laga, 29071-M{\'a}laga, Spain}
\email{david.mendez@uma.es}
\thanks{The second author was partially supported by Ministerio de Educaci\'on, Cultura y Deporte grant FPU14/05137, and by Ministerio de Econom\'ia y Competitividad (Spain) grants MTM2016-79661-P and MTM2016-78647-P}
\author{Antonio Viruel}
\address{Departamento de \'Algebra, Geometr{\'\i}a y Topolog{\'\i}a, Universidad de M{\'a}laga, 29071-M{\'a}laga, Spain}
\email{viruel@uma.es}
\thanks{The third author was partially supported by by Ministerio de Econom\'ia y Competitividad (Spain) grant MTM2016-78647-P}


\begin{abstract}
Let $X$ be a finite type $A_n^2$-polyhedron, $n \geq 2$. In this paper we study the quotient group $\E(X)/\E_*(X)$, where $\E(X)$ is the group of self-homotopy equivalences of $X$ and $\E_*(X)$ the subgroup  of self-homotopy equivalences inducing the identity on the homology groups of $X$. We show that not every group can be realised as $\E(X)$ or $\E(X)/\E_*(X)$ for $X$ an $A_n^2$-polyhedron, $n\ge 3$, and specific results are obtained for $n=2$.
\end{abstract}

\maketitle


\section{Introduction}\label{sec:intro}

Let $\E(X)$ denote the group of homotopy classes of self-homotopy equivalences
of  a space $X$ and $\E_*(X)$ denote the normal subgroup of self-homotopy equivalences inducing the identity on the homology groups of $X$. Problems related to $\E(X)$ have been extensively studied, deserving a special mention Kahn's realisability problem, which has been placed first to solve in \cite{Ark01} (see also \cite{Ark90, Kah76, Kah90, Rut97}). It asks whether an arbitrary group  can be realised as $\mathcal E(X)$ for some  simply connected $X$, and though the general case remains an open question, it has recently been solved for finite groups, \cite{CosVir14}. As a way to approach Kahn's problem, in \cite[Problem 19]{Fel10} the question of whether an arbitrary group can appear as the distinguished quotient $\E(X)/\E_*(X)$ is raised.

In this paper we work with $(n-1)$-connected $(n+2)$-dimensional $CW$-complexes  for $n\ge 2$, the so-called $A_n^2$-polyhedra. Homotopy types of these spaces have been classified by Baues in \cite[Ch. I, \S8]{Bau91}  using the long exact sequence of groups associated to simply connected spaces introduced by J.\ H.\ C.\ Whitehead in \cite{Whi50}. In \cite{Ben16}, the author uses that classification to study the group of self-homotopy equivalences of an $A_2^2$-polyhedron $X$.  He associates to $X$ a group $\B^4(X)$ that is isomorphic to $\E(X)/\E_*(X)$ and asks if any group can be realised as such a quotient in this context,  that is, if $A_2^2$-polyhedra provide an adequate framework to solve the realisability problem.

Here, in the general setting of an $A_n^2$-polyhedra $X$, $n\geq 2$, we also construct a group $\B^{n+2}(X)$ (see Definition \ref{def:BnX}) that  is isomorphic to $\E(X)/\E_*(X)$ (see Proposition \ref{prop:selfequiviso}). We show that there exist many groups (for example $\mathbb Z/p$, $p$ odd, Corollary \ref{cor:notrealisability}) for which the question above  does not admit a positive answer. This fact should illustrate that $A_n^2$-polyhedra might  not be the right setting to answer  \cite[Problem 19]{Fel10}.

We show, for instance, that under some restrictions on the homology groups of $X$, $\B^{n+2}(X)$ is infinite, which in particular implies that $\E(X)$ is infinite (see Proposition \ref{prop:mainfinite} and Proposition \ref{prop:rankfreeHn}). Or for example, in many situations the existence of odd order elements in the homology groups of  $X$ implies the existence of involutions in $\B^{n+2}(X)$ (see Lemma \ref{lemma:Hnnot2group} and Lemma \ref{lemma:homologynot2group}).

In this paper we prove the following result:
\begin{theorem}\label{th:realisabilitygr3}
	Let $X$ be a finite type $A_n^2$-polyhedron, $n\ge 3$. Then $\B^{n+2}(X)$ is either the trivial group or it has elements of even order.	
\end{theorem}

As an immediate corollary, we obtain the following:
\begin{corollary} \label{cor:notrealisability} Let $G$ be a non-trivial group with no elements of even order. Then $G$ is not realisable as $\B^{n+2}(X)$ for $X$ a finite type $A_n^2$-polyhedron, $n\ge 3$.
	
	\end{corollary}

The case $n=2$ is more complicated.  A detailed group theoretical analysis shows that a finite type $A_2^2$-polyhedra might realise finite groups of odd order only  under very restrictive conditions. Recall that for a group $G$, $\rank G$ is the smallest cardinal of a set of generators for $G$ \cite[p.\ 91]{LS2001}. We have the following result:
\begin{theorem}\label{th:realisabilityn2}
	Suppose that $X$ is a finite type $A_2^2$-polyhedron with a non trivial finite  $ \B^4(X)$  of odd order. Then the following holds:
	\begin{enumerate}
		\item $\rank H_4(X) \le 1$,
		\item $\pi_3(X)$ and $H_3(X)$ are $2$-groups, and $H_2(X)$ is an elementary abelian $2$-group,
		\item $\rank H_3(X)  \le \frac{1}{2}\rank H_2(X)\big(\rank H_2(X) + 1\big) - \rank H_4(X)\le \rank\pi_3(X)$,
		\item the natural action of $\B^4(X)$ on $H_2(X)$ induces a faithful representation $\B^4(X) \le \Aut\big(H_2(X)\big)$.
	\end{enumerate}
\end{theorem}
All our attempts to find a space satisfying the hypothesis of Theorem \ref{th:realisabilityn2} were unsuccessful. We therefore raise the following conjecture:
\begin{conjecture}\label{question:realisability}
Let $X$ be an $A_2^2$-polyhedron. If $\B^4(X)$ is a non trivial finite group, then it necessarily has an element of even order.
\end{conjecture}

This paper is organised as follows. In Section \ref{sec:classification} we give a brief introduction to Whitehead and Baues results for the classification of homotopy types of $A_n^2$-polyhedra, or equivalently, isomorphism classes of certain long exact sequences of abelian groups (see Theorem \ref{th:detectingfunctor}).
In Section \ref{sec:generalresults} we study how restrictions on $X$ affect the group $\B^{n+2}(X)$.
Finally, Section \ref{sec:obstructions} is devoted to the proof of our main results, Theorem \ref{th:realisabilitygr3} and Theorem \ref{th:realisabilityn2}.


\section{The $\Gamma$-sequence of an $A_n^2$-polyhedron}\label{sec:classification}

Let \text{Ab} denote the category of abelian groups. In \cite{Whi50}, J.H.C. Whitehead constructed a functor $\Gamma\colon \text{Ab} \to \text{Ab}$, known as the Whitehead's universal quadratic functor, and an exact sequence, which are useful to our purposes and we introduce in this section. The $\Gamma$-functor is defined as follows. Let $A$ and $B$ be abelian groups and $\eta\colon A\to B$ be a map (of sets) between them. The map $\eta$ is said to be quadratic if:
\begin{enumerate}
\item $\eta(a)=\eta(-a)$, for all $a\in A$, and
\item the map $A\times A\to B$ taking $(a, a')$ to $\eta(a+a')-\eta(a)-\eta(a')$ is bilinear.
\end{enumerate}

For an abelian group $A$, $\Gamma(A)$ is the only abelian group such that there exists a quadratic map $\gamma\colon A\to \Gamma(A)$ verifying that every other quadratic map $\eta \colon A\to B$ factors uniquely through $\gamma$. This means that there is a unique group homomorphism $\eta^\square\colon \Gamma(A)\to B$ such that $\eta=\eta^\square \gamma$. The quadratic map $\gamma\colon A\to \Gamma(A)$ receives the name of universal quadratic map of $A$.

The $\Gamma$-functor acts on morphisms as follows: let $f\colon A\to B$ be a group homomorphism, and $\gamma\colon A\to\Gamma(A)$ and $\gamma\colon B\to \Gamma(B)$ the universal quadratic maps. Then, $\gamma f\colon A\to \Gamma(B)$ is a quadratic map, so there exists a unique group homomorphism $(\gamma f)^\square\colon \Gamma(A)\to \Gamma(B)$ such that $(\gamma f)^\square\gamma = \gamma f$. Define $\Gamma(f)=(\gamma f)^\square$.

We now list some of its properties that will be used later in this paper:
\begin{proposition}{\rm(\cite[pp.\ 16--17]{Bau96})}\label{prop:propertiesgammafunctor}
The $\Gamma$ functor has the following properties:
\begin{enumerate}
\item $\Gamma(\Z)=\Z$,
\item $\Gamma(\Z_n)$ is $\Z_{2n}$ if $n$ is even or $\Z_n$ if $n$ is odd,
\item Let $I$ be an ordered set and $A_i$ be an abelian group, for each $i\in I$. Then, $${\Gamma\big(\bigoplus_I A_i\big) = \big(\bigoplus_I \Gamma(A_i)\big) \oplus \big(\bigoplus_{i<j} A_i\otimes A_j\big)}.$$ Moreover, the groups $\Gamma(A_i)$ and $A_i\otimes A_j$ are respectively generated by elements $\gamma(a_i)$ and $a_i\otimes a_j$, with $a_i\in A_i$, $a_j\in A_j$, $i < j$, and  $\gamma(a_i+a_j)=\gamma(a_i)+\gamma(a_j)+a_i\otimes a_j$, for $a_i\in A_i$, $a_j\in A_j$, $i<j$, \cite[\textsection 5, \textsection 7]{Whi50}.
\end{enumerate}
\end{proposition}

We now introduce Whitehead's exact sequence. Let $X$ be a  simply connected $CW$-complex. For $n\ge 1$, the $n$-th Whitehead $\Gamma$-group of $X$ is defined as $$\Gamma_n(X)=\Im\big(i_*\colon \pi_n (X^{n-1}) \to \pi_n (X^n)\big).$$
Here, $i\colon X^{n-1}\to X^n$ is the inclusion of the $(n-1)$-skeleton of $X$ into its $n$-skeleton. Then, $\Gamma_n(X)$ is an abelian group for $n\ge 1$. This group can be embedded into a long exact sequence of abelian groups
\begin{equation}\label{eq:whiteheadseq}
\cdots\to H_{n+1}(X)\xrightarrow{b_{n+1}} \Gamma_n(X) \xrightarrow{i_{n-1}} \pi_{n}(X)\xrightarrow{h_{n}} H_{n}(X)\to \cdots
\end{equation}
where $h_n$ is the Hurewicz homomorphism and $b_{n+1}$ is a boundary representing the attaching maps.

For each $n\ge 2$, a functor $\Gamma_n^1\colon\text{Ab}\to \text{Ab}$ is defined as follows. Let $\Gamma_2^1 = \Gamma$ be the universal quadratic functor, and for $n\ge 3$, $\Gamma_n^1 = -\otimes \Z_2$. It turns out that if $X$ is $(n-1)$-connected, then $\Gamma_n^1\big(H_n(X)\big)\cong \Gamma_{n+1}(X)$, \cite[Theorem 2.1.22]{Bau96}. Thus, the final part of the long exact sequence \eqref{eq:whiteheadseq} can be written as
\begin{equation}\label{eq:gammaseq}
H_{n+2}(X)\xrightarrow{b_{n+2}} \Gamma_n^1\big(H_n(X)\big)\xrightarrow{i_n} \pi_{n+1}(X)\xrightarrow{h_{n+1}} H_{n+1}(X)\to 0.
\end{equation}

Now, for each $n\ge 2$, we define the category of $A_n^2$-polyhedra as the category whose objects are $(n+2)$-dimensional $(n-1)$-connected $CW$-complexes and whose morphisms are continuous maps between objects. Homotopy types of these spaces are classified through isomorphism classes in a category whose objects are sequences like \eqref{eq:gammaseq}, \cite[Ch. I, \textsection 8]{Bau91}:
\begin{definition}\label{def:gammasequences}(\cite[Ch. IX, \textsection 4]{Bau89})
	Let $n\ge 2$ be an integer. We define the category of $\Gamma$-sequences$^{n+2}$ as follows. Objects are exact sequences of abelian groups
	\[H_{n+2}\to \Gamma_n^1(H_n)\to \pi_{n+1}\to H_{n+1}\to 0\]
	where $H_{n+2}$ is free abelian. Morphisms are triples of group homomorphisms $f=(f_{n+2},f_{n+1},f_n)$, $f_i\colon H_i\to H'_i$, such that there exists a group homomorphism $\Omega\colon \pi_{n+1}\to \pi'_{n+1}$  making the following diagram
	\begin{center}
\begin{tikzpicture}
\tikzset{node distance=0.15\textwidth, auto}
\node(1) {$H_{n+2}$};
\node[right of=1] (2) {$\Gamma_n^1(H_n)$};
\node[right of=2] (3) {$\pi_{n+1}$};
\node[right of=3] (4) {$H_{n+1}$};
\node[right of=4] (5) {$0$};
\node[below of=1, yshift=0.07cm] (6) {$H'_{n+2}$};
\node[right of=6] (7) {$\Gamma_n^1(H'_n)$};
\node[right of=7] (8) {$\pi'_{n+1}$};
\node[right of=8] (9) {$H'_{n+1}$};
\node[right of=9] (10) {$0$};
\draw[->](1) to node {} (2);
\draw[->](2) to node {} (3);
\draw[->](3) to node {} (4);
\draw[->](4) to node {} (5);
\draw[->](6) to node{} (7);
\draw[->](7) to node {} (8);
\draw[->](8) to node {} (9);
\draw[->](9) to node {} (10);
\draw[->](1) to node {\small $f_{n+2}$} (6);
\draw[->] (2) to node {\small $\Gamma_n^1(f_n)$} (7);
\draw[->] (3) to node {\small $\Omega$} (8);
\draw[->] (4) to node {\small $f_{n+1}$} (9);
\end{tikzpicture}
\end{center}
commutative. We say that objects in $\Gamma\text{-sequences}^{n+2}$ are $\Gamma$-sequences, and morphisms in the category are called $\Gamma$-morphisms.
\end{definition}

On the one hand, we can assign to an $A_n^2$-polyhedron $X$, an object in $\Gamma$-sequences$^{n+2}$ by considering the associated exact sequence, \eqref{eq:gammaseq}. We call such an object the \emph{$\Gamma$-sequence of $X$}. On the other hand, to a continuous map $\alpha\colon X\to X'$ of $A_n^2$-polyhedra we can assign a morphism between the corresponding $\Gamma$-sequences by considering the induced homomorphisms
\begin{center}
\begin{tikzpicture}
\tikzset{node distance=0.15\textwidth, auto}
\node(1) {$H_{n+2}(X)$};
\node[right of=1, xshift=0.5cm] (2) {$\Gamma_n^1\big(H_n(X)\big)$};
\node[right of=2, xshift=0.6cm] (3) {$\pi_{n+1}(X)$};
\node[right of=3, xshift=0.4cm] (4) {$H_{n+1}(X)$};
\node[right of=4] (5) {$0$};
\node[below of=1] (6) {$H_{n+2}(X')$};
\node[right of=6, xshift=0.5cm] (7) {$\Gamma_n^1\big(H_n(X')\big)$};
\node[right of=7, xshift=0.6cm] (8) {$\pi_{n+1}(X')$};
\node[right of=8, xshift=0.4cm] (9) {$H_{n+1}(X')$};
\node[right of=9] (10) {$0$.};
\draw[->](1) to node {} (2);
\draw[->](2) to node {} (3);
\draw[->](3) to node {} (4);
\draw[->](4) to node {} (5);
\draw[->](6) to node {} (7);
\draw[->](7) to node {} (8);
\draw[->](8) to node {} (9);
\draw[->](9) to node {} (10);
\draw[->](1) to node {\small $H_{n+2}(\alpha)$} (6);
\draw[->] (2) to node {\small $\Gamma_n^1\big(H_n(\alpha)\big)$} (7);
\draw[->] (3) to node {\small $\pi_{n+1}(\alpha)$} (8);
\draw[->] (4) to node {\small $H_{n+1}(\alpha)$} (9);
\end{tikzpicture}
\end{center}
Therefore, we have a functor $A_n^2\text{-polyhedra}\to \Gamma\text{-sequences}^{n+2}$ which clearly restricts to the homotopy category of $A_n^2$-polyhedra, $\mathcal{H}oA_n^2\text{-polyhedra}$. It is obvious that this functor sends homotopy equivalences to isomorphisms between the corresponding $\Gamma$-sequences. Thus, we can classify homotopy types of $A_n^2$-polyhedra through isomorphism classes of $\Gamma$-sequences:

\begin{theorem}[{\cite[Ch. I, \textsection 8]{Bau91}}]\label{th:detectingfunctor}
	The functor $\mathcal{H}oA_n^2\emph{-polyhedra}\to \Gamma\emph{-sequences}^{n+2}$ previously defined is full. Moreover, for any object in $\Gamma\emph{-sequences}^{n+2}$, there exists an $A_n^2$-polyhedron whose $\Gamma$-sequence is the given object in $\Gamma\emph{-sequences}^{n+2}$. In fact, there exists a 1--1 correspondence between homotopy types of $A_n^2$-polyhedra and isomorphism classes of $\Gamma$-sequences.
\end{theorem}

 Following the ideas of \cite{Ben16}, we introduce the following:
\begin{definition}\label{def:BnX}
Let $X$ be an $A_n^2$-polyhedron. We denote by $\B^{n+2}(X)$ the group of $\Gamma$-isomorphisms of the $\Gamma$-sequence of $X$.
\end{definition}

Let $\Psi\colon \E(X)\to \B^{n+2}(X)$ be the map that associates to $\alpha\in\E(X)$ the $\Gamma$-isomorphism $\Psi(\alpha)=\big(H_{n+2}(\alpha), H_{n+1}(\alpha), H_n(\alpha)\big)$. Then $\Psi$ is a group homomorphism: its kernel is the subgroup of  self-homotopy equivalences inducing the identity map on the homology groups of $X$, that is, $\E_*(X)$. Also, $\Psi$ is onto as a consequence of Theorem \ref{th:detectingfunctor}. Hence,  we immediately obtain the following result.

\begin{proposition}\label{prop:selfequiviso}
	Let $X$ be an $A_n^2$-polyhedron, $n\ge 2$. Then $\B^{n+2}(X) \cong \E(X)/\E_*(X)$.
\end{proposition}


\section{Self-homotopy equivalences of finite type $A_n^2$-polyhedra}\label{sec:generalresults}

Henceforth, an $A_n^2$-polyhedron will mean an $(n-1)$-connected, $(n+2)$-dimensional $CW$-complex of finite type. Recall that for simply connected and finite type spaces, the homology and homotopy groups $H_n(X)$ and $\pi_n(X)$ are finitely generated and abelian for $n\geq 1$.

The $\Gamma$-sequence tool introduced in Section \ref{sec:classification}  will help us to illustrate, from an algebraic point of view, how different restrictions on an $A_n^2$-polyhedron $X$ affect the quotient group $\E(X)/\E_*(X)$. We devote this section to that matter. We also obtain several results that are needed in the proof of Theorem \ref{th:realisabilitygr3} and Theorem \ref{th:realisabilityn2}.
The following result is a generalisation of \cite[Theorem 4.5]{Ben16}.

\begin{proposition}\label{prop:hurewiczonto}
Let $X$ be an $A_n^2$-polyhedron and suppose that the Hurewicz homomorphism $h_{n+2}\colon \pi_{n+2}(X)\to H_{n+2}(X)$ is onto. Then, every automorphism of $H_{n+2}(X)$ is realised by a self-homotopy equivalence of $X$.
\end{proposition}

\begin{proof}
As part of the exact sequence \eqref{eq:whiteheadseq} for $X$ we have:
\[\cdots \to \pi_{n+2}(X) \xrightarrow{h_{n+2}} H_{n+2}(X)\xrightarrow{b_{n+2}} \Gamma_n^1\big(H_n(X)\big) \to \pi_{n+1}(X) \to \cdots\]
Then, since $h_{n+2}$ is onto by hypothesis, $b_{n+2}$ is the trivial homomorphism. Thus, for every $f_{n+2}\in\Aut(H_{n+2}(X))$, $b_{n+2} f_{n+2} = b_{n+2} = 0$, so if $\Omega= \id$, $(f_{n+2},\id,\id)\in \B^{n+2}(X)$. Then there exists $f \in \E(X)$ with $H_{n+2}(f)=f_{n+2}$, $H_{n+1}(f)=\id$, $H_n(f)=\id$.
\end{proof}

We can easily prove that automorphism groups can be realised, a result that can also be obtained as a consequence of \cite[Theorem 2.1]{Rut97}:

\begin{example}\label{ex:GautH}
Let $G$ be a group isomorphic to $\Aut(H)$ for some finitely generated abelian group $H$. Then, for any integer $n\ge 2$, there exists an $A_n^2$-polyhedron $X$ such that $G \cong \B^{n+2}(X)$: take the Moore space $X=M(H, n+1)$,  which in particular is an $A_n^2$-polyhedron. The  $\Gamma$-sequence of $X$ is
\[H_{n+2}(X)=0\to \Gamma_n^1\big(H_n(X)\big)=0\to H\xrightarrow{=} H\to 0.\]
Then, for every $f\in \Aut(H)$, by taking $\Omega=f$ we see that $(\id, f,\id)\in \B^{n+2}(X)$, and those are the only possible $\Gamma$-isomorphisms. Thus $\B^{n+2}(X)\cong \Aut(H) \cong  G$.
\end{example}

The use of Moore spaces is not required in the $n=2$ case:
\begin{example}\label{ex:notmoore}
	 Let $G$ be a group isomorphic to $\Aut(H)$ for some finitely generated abelian group $H$. Consider the following object in $\Gamma$-sequences$^4$
	\begin{equation}\label{eq:gammaseqnotmoore}
		\Z\xrightarrowdbl{b_4} \Gamma(\Z_2)=\Z_4\to H\xrightarrow{=} H\to 0.
	\end{equation}
	By Theorem \ref{th:detectingfunctor}, there exists an $A_2^2$-polyhedron $X$ realising this object. In particular, $H_4(X)=\Z$, $H_3(X)=\pi_3(X)=H$ and $H_2(X)=\Z_2$. It is clear from \eqref{eq:gammaseqnotmoore} that $(\id, f,\id)$ is a $\Gamma$-isomorphism for every $f\in \Aut(H)$. Now $\Aut(\Z_2)$ is the trivial group while $\Aut(\Z)=\{-\id,\id\}$. It is immediate to check that $(-\id, f,\id)$ is not a $\Gamma$-isomorphism since $\id b_4\ne b_4(-\id)$. Then, we obtain that $\B^4(X)\cong \Aut(H)$.
\end{example}

Observe that not every group $G$ is isomorphic to the automorphism group of an abelian group (for example $\Z_p$ if $p$ is odd). Hence, examples from above only provide a partial positive answer to the realisability problem for $\B^{n+2}(X)$.  Indeed, the automorphism group of an abelian group (other than $\mathbb Z_2$)  has elements of even order. The following results go in that direction:

\begin{lemma}\label{lemma:Hnnot2group}
Let $X$ be an $A_n^2$-polyhedron, $n\ge 2$. If $H_n(X)$ is not an elementary abelian $2$-group, then $\B^{n+2}(X)$ has an element of order $2$.
\end{lemma}

\begin{proof}
Since $H_n(X)$ is not an elementary abelian $2$-group, it admits a non-trivial involution $-\id\colon H_n(X)\to H_n(X)$. But $\Gamma_n^1(-\id)=\id$ for every $n\ge 2$, so $(\id,\id,-\id)\in \B^{n+2}(X)$ and the result follows.
\end{proof}

We point out a key difference between the $n=2$ and the $n\ge 3$ cases:  $\Gamma_2^1(A)=\Gamma(A)$ is never an elementary abelian $2$-group when $A$ is  finitely generated and abelian, as it can be deduced from Proposition \ref{prop:propertiesgammafunctor}. However, for $n\ge 3$, $\Gamma_n^1(A)=A\otimes\Z_2$ is always an elementary abelian $2$-group. Taking advantage of this fact we can prove the following result:

\begin{lemma}\label{lemma:homologynot2group}
Let $X$ be an $A_n^2$-polyhedron, $n\ge 3$. If any of the homology groups of $X$ is not an elementary abelian $2$-group (in particular, if $H_{n+2}(X)\ne 0$), then $\B^{n+2}(X)$ contains a non trivial element of order $2$.
\end{lemma}

\begin{proof}
Under our assumptions, $\Gamma_n^1\big(H_n(X)\big)$ is an elementary abelian $2$-group. For  $\Omega=-\id$, the triple $(-\id,-\id,-\id)$ is a $\Gamma$-isomorphism of order $2$ unless $H_{n+2}(X)$, $H_{n+1}(X)$ and $H_n(X)$ are all elementary abelian $2$-groups.
\end{proof}

We remark that this result does not hold for $A_2^2$-polyhedra. Indeed, if we consider the construction in Example \ref{ex:notmoore} for $H=\Z_2$, then $\B^4(X)\cong \Aut(\Z_2)=\{*\}$ does not contain a non trivial element of order $2$ although $H_4(X)=\Z$ is not an elementary abelian $2$-group.

We now prove some results regarding the finiteness of $\B^{n+2}(X)$:
\begin{proposition}\label{prop:mainfinite}
Let $X$ be an $A_n^2$-polyhedron, $n\geq 2$, with $\rank H_{n+2}(X)\ge 2$ and every element of $\Gamma_n^1\big(H_n(X)\big)$ of finite order. Then $\B^{n+2}(X)$ is an infinite group.
\end{proposition}

\begin{proof}
Since $\rank H_{n+2}(X)\ge 2$, we may write $H_{n+2}(X)=\Z^2\oplus G$, $G$ a  (possibly trivial) free abelian group. Consider the $\Gamma$-sequence of $X$:
\[\Z^2\oplus G \xrightarrow{b_{n+2}} \Gamma_n^1\big(H_n(X)\big) \xrightarrow{i_n} \pi_{n+1}(X) \xrightarrow{h_{n+1}} H_{n+1}(X)\longrightarrow 0.\]
Since $b_{n+2}(\Z^2)\le \Gamma_n^1\big(H_n(X)\big)$ is a finitely generated $\Z$-module with finite order generators, it is a finite group. Define $k = \operatorname{exp}\big(b_{n+2}(\Z^2)\big)$ and consider the automorphism of $\Z^2$ given by the matrix
\[\begin{pmatrix}
1 & k \\ 0 & 1	
\end{pmatrix}\in \GL_2(\Z),
\]
which is of infinite order. If we take $f\oplus \id_G\in \Aut(\Z^2\oplus G)$, then $b_{n+2}(f\oplus \id)=b_{n+2}$, thus $(f\oplus \id_G,\id,\id)\in \B^{n+2}(X)$, which is an element of infinite order.
\end{proof}

As we have previously mentioned, $\Gamma_n^1\big(H_n(X)\big)$ is an elementary abelian $2$-group, for $n\ge 3$. Hence, from Proposition \ref{prop:mainfinite} we get:
\begin{corollary}
Let $X$ be an $A_n^2$-polyhedron, $n\ge 3$, with $\rank H_{n+2}(X)\ge 2$. Then $\B^{n+2}(X)$ is an infinite group.
\end{corollary}

This result does not hold, in general, for $n=2$. However, if $A$ is a finite group, Proposition \ref{prop:propertiesgammafunctor} implies that $\Gamma(A)$ is finite as well so from Proposition \ref{prop:mainfinite} we get:

\begin{corollary}
Let $X$ be an $A_2^2$-polyhedron  with $\rank H_{4}(X)\ge 2$ and $H_2(X)$ finite. Then $\B^{4}(X)$ is an infinite group.
\end{corollary}

We end this section with one more result on the infiniteness of $\B^{n+2}(X)$:
\begin{proposition}\label{prop:rankfreeHn}
Let $X$ be an $A_n^2$-polyhedron, $n\ge 3$. If $H_n(X)=\Z^2\oplus G$ for a certain abelian group $G$, then $\B^{n+2}(X)$ is an infinite group.
\end{proposition}

\begin{proof}
If $H_n(X)=\Z^2\oplus G$, then $\Gamma_n^1\big(H_n(X)\big) = H_n(X)\otimes \Z_2 = \Z_2^2\oplus (G\otimes \Z_2)$. Hence $\GL_2(\Z)\le \Aut\big(H_n(X)\big)$ and $\GL_2(\Z_2)\le \Aut\big(H_n(X)\otimes \Z_2\big)$. Moreover, for every $f\in \GL_2(\Z)$ we have $f\oplus \id_G\in \Aut\big(H_n(X)\big)$ which yields, through $\Gamma_n^1$, an automorphism $(f\oplus \id_G)\otimes \Z_2 = (f\otimes \Z_2)\oplus \id_{G\otimes \Z_2}\in \Aut\big(H_n(X)\otimes \Z_2\big)$. This means that the functor $\Gamma_n^1$ restricts to $\GL_2(\Z)\to \GL_2(\Z_2)$. Moreover, $-\otimes \Z_2 \colon \GL_2(\Z)\to \GL_2(\Z_2)$ has an infinite kernel. Hence, there are infinitely many morphisms $f\in \Aut\big(H_n(X)\big)$ such that $f\otimes \Z_2=\id$. For any such a morphism $f$,  $(\id,\id, f)$ is an element of $\B^{n+2}(X)$. Therefore $\B^{n+2}(X)$ is infinite.
\end{proof}


\section{Obstructions to the realisability of groups}\label{sec:obstructions}

We have seen in Section \ref{sec:generalresults} that the group  $\B^{n+2}(X)$ contains elements of even order unless strong restrictions are imposed on the homology groups of the $A_n^2$-polyhedron $X$. Since we are interested in realising an arbitrary group $G$ as $\B^{n+2}(X)$ for $X$ a finite-type $A_n^2$-polyhedron, in this section we focus our attention on the remaining situations and prove Theorems \ref{th:realisabilitygr3} and \ref{th:realisabilityn2}.
We first give some previous results:

\begin{lemma}\label{lemma:gammainjective}
For $G$ an elementary abelian $2$-group, $\Gamma(-)\colon \Aut(G)\to\Aut\big(\Gamma(G)\big)$ is injective.
\end{lemma}

\begin{proof}
Let us show that the kernel of $\Gamma(-)$ is trivial. Assume that $G$ is generated by $\{e_j\mid j\in J\}$, $J$ an ordered set. If $f\in \Aut(G)$ is in the kernel of $\Gamma(-)$, then for each $j\in J$, there exists a finite subset $I_j\subset J$ such that $f(e_j)= \sum_{i\in I_j}  e_i$, and
	 \[\gamma(e_j)=\Gamma(f)\gamma(e_j)=\gamma f(e_j)= \gamma\left( \sum_{i\in I_j} e_i \right) = \sum_{i\in I_j} \gamma(e_i)+\sum_{i<k} e_i\otimes e_k,\]
	 as a consequence of Proposition \ref{prop:propertiesgammafunctor}.(3), so $I_j=\{j\}$ and $f(e_j)=e_j$ for every $j\in J$.
 \end{proof}

\begin{lemma}\label{lemma:h3splitIfOddThenEven}
	Let $H_2 = \oplus_{i=1}^n\mathbb{Z}_2$ and $\chi\in\Gamma(H_2)$ be an element of order $4$. If there exists a non trivial automorphism of odd order $f\in\Aut(H_2)$ such that $\Gamma(f)(\chi) =\chi$, then there exists $g\in\Aut(H_2)$ of order $2$ such that $\Gamma(g)(\chi) =\chi$.
\end{lemma}

\begin{proof}
	Notice that according to \cite[p.\ 66]{Whi50}, we can write $h\otimes h = 2\gamma(h)$, for any element $h\in H_2$. Therefore, given a basis $\{h_1,h_2,\dots,h_n\}$ of $H_2$, and replacing $3\gamma(h_i)$ by $\gamma(h_i)+ h_i\otimes h_i$ if needed, we can write
$$\chi = \sum_{i=1}^n a(i) \gamma(h_i) + \sum_{i,j=1}^n a(i,j) h_i\otimes h_j,$$
where every coefficient $a(i)$, $a(i,j)$ is either $0$ or $1$. We now construct inductively a basis $\{e_1,e_2,\dots,e_n\}$ of $H_2$ as follows. Without loss of generality, assume $a(1)=1$ and define $e_1=\sum_{i=1}^n a(i) h_i$. Then $\{e_1,h_2,\dots,h_n\}$ is again a basis of $H_2$ and
$$
\chi = \gamma(e_1) + \alpha_1 e_1\otimes e_1 + \\
\beta_1 e_1\otimes(\sum_{s=2}^n b(1,s) h_s) + \sum_{i,j>1}^n a_1(i,j) h_i\otimes h_j,
$$
where every coefficient in the equation is either $0$ or $1$.
Assume a basis $\{e_1,\ldots, e_r, h_{r+1},\dots,h_n\}$ has been constructed such that
\begin{align*}
\chi =&\, \gamma(e_1) + \sum_{j=1}^r \alpha_j e_j\otimes e_j + \sum_{j=1}^{r-1} \beta_j e_j\otimes e_{j+1}\\ &+ \beta_{r} e_{r}\otimes\big(\sum_{s=r+1}^n b(r,s) h_s\big) + \sum_{i,j>r}^n a_r(i,j) h_i\otimes h_j,
\end{align*}
where every coefficient in the equation is either $0$ or $1$. We may assume $b(r,r+1)=1$ and define $e_{r+1}=\sum_{s=r+1}^n b(r,s) h_s$. Thus $\{e_1,\ldots, e_{r+1}, h_{r+2},\dots,h_n\}$ is again a basis of $H_2$ and
\begin{align*}
\chi = &\, \gamma(e_1) + \sum_{j=1}^{r+1} \alpha_j e_j\otimes e_j + \sum_{j=1}^{r} \beta_j e_j\otimes e_{j+1}\\ &+ \beta_{r+1} e_{r+1}\otimes\big(\sum_{s=r+2}^n b(r+1,s) h_s\big) + \sum_{i,j>r+1}^n a_{r+1}(i,j) h_i\otimes h_j.
\end{align*}
Finally, we obtain a basis $\{e_1,e_2,\dots,e_n\}$ of $H_2$ such that
	\begin{equation}\label{eq:b4decompfinal}
	\chi = \gamma(e_1) + \sum_{j=1}^n \alpha_j e_j\otimes e_j + \sum_{j=1}^{n-1} \beta_j e_j\otimes e_{j+1},
	\end{equation}
for some coefficients $\alpha_j\in\{0,1\}$, $j=1,2,\dots,n$, and $\beta_j\in\{0,1\}$, $j = 1,2,\dots,n-1$.

	Now, for $n=1$, $H_2 = \mathbb{Z}_2$ has a trivial group of automorphisms, so the result  holds. For $n = 2$, assume that there exists $f\in\Aut(H_2)$ such that $\Gamma(f)(\chi) =\chi$. From Equation \eqref{eq:b4decompfinal}, $\chi = \Gamma(f)\big(\gamma(e_1)\big) + \Gamma(f)(P)$, where $P\in\Omega_1\big(\Gamma(H_2)\big) = \{h\in\Gamma(H_2) : \ord(h)|2\}$. Then $\Gamma(f)\big(\gamma(e_1)\big)$ has a multiple of $\gamma(e_1)$ as its only summand of order $4$, which implies that $f(e_1) = e_1$. Then either $f(e_2) = e_2$, so $f$ is trivial, or $f(e_2) = e_1 + e_2$, so $f$ has order $2$.

For $n\ge 3,$
we define $g\in\Aut(H_2)$ by $g(e_j) = e_j$, for $j = 1,2,\dots,n-2$, and $g(e_{n-1})$ and $g(e_n)$,  depending on $\alpha_{n-j}$ and $\beta_{n-1-j}$, for $j=0,1$,  in Equation \eqref{eq:b4decompfinal},  according to the following table:

	\[\begin{array}{|c|c|c|c|c|c|}
		\hline
		\alpha_n & \beta_{n-1} & \alpha_{n-1} & \beta_{n-2}  &g(e_{n-1}) &g(e_n) \\\hline
		0 & 0 & \text{$0$ or $1$} & \text{$0$ or $1$} & e_{n-1} & e_{n-1} + e_n \\\hline
		0 & 1 & 0 & 0 & e_n                     & e_{n-1} \\\hline
		0 & 1 & 0 & 1 & e_{n-2} + e_n           & e_{n-2} + e_{n-1} \\\hline
		0 & 1 & 1 & 0 & e_{n-1} + e_n           & e_n \\\hline
		0 & 1 & 1 & 1 & e_{n-2} + e_{n-1} + e_n & e_n \\\hline
		1 & 0 & 0 & 0 & e_{n-2} + e_{n-1}       & e_n \\\hline
		1 & 0 & 0 & 1 & e_{n-2} + e_{n-1}       & e_{n-2} + e_n \\\hline
		1 & 0 & 1 & 0 & e_n                     & e_{n-1} \\\hline
		1 & 0 & 1 & 1 & e_{n-2} + e_{n-1}       & e_n \\\hline
		1 & 1 & \text{$0$ or $1$} & \text{$0$ or $1$} & e_{n-1} & e_{n-1} + e_n \\\hline
	\end{array}\]
A simple computation shows that in all cases $g$ has order $2$ and $\Gamma(g)(\chi) = \chi$, so the result follows.
\end{proof}

\begin{definition}
	Let $f\colon H\xrightarrow{} K$ be a morphism of abelian groups. We say that a non-trivial subgroup $A\le K$ is $f$-split if there exist groups $B\le H$ and $C\le K$ such that $H\cong A\oplus B$, $K= A\oplus C$ and $f$ can be written as $\id_A\oplus g\colon A\oplus B\to A\oplus C$ for some $g\colon B\to C$.
\end{definition}
Henceforward we will make extensive use of this notation applied to $ h_{n+1}\colon \pi_{n+1}(X) \xrightarrow{} H_{n+1}(X) $, the Hurewicz morphism.
We prove the following:

\begin{lemma}\label{lemma:autsplitsummands}
Let $X$ be an $A_n^2$-polyhedron, $n\ge 2$. Let $A\leq H_{n+1}(X)$ be an $h_{n+1}$-split subgroup, thus $H_{n+1}(X)=A\oplus C$ for some abelian group $C$.
Then, for every $f_A\in\Aut(A)$ there exists
$f\in\E(X)$ inducing $(\id, f_A\oplus \id_C, \id) \in \B^{n+2}(X)$.
\end{lemma}

\begin{proof}
By hypothesis $H_{n+1}(X)=A\oplus C$, $\pi_{n+1}(X)\cong A\oplus B$,  for some abelian group $B$, and  $h_{n+1} $ can be written as  $\id_A\oplus g$ for some morphism $g\colon B\to C$. Thus, for every $f_A\in \Aut(A)$ we have a commutative diagram
\begin{center}
\begin{tikzpicture}
\tikzset{node distance=0.15\textwidth, auto}
\node(1) {$H_{n+2}(X)$};
\node[right of=1, xshift=0.9cm] (2) {$\Gamma_n^1\big(H_n(X)\big)$};
\node[right of=2, xshift=0.5cm] (3) {$A \oplus B$};
\node[right of=3, xshift=0.8cm] (4) {$A \oplus C$};
\node[right of=4] (5) {$0$};
\node[below of=1, yshift=0.06cm] (6) {$H_{n+2}(X)$};
\node[below of=2, yshift=0.06cm] (7) {$\Gamma_n^1\big(H_n(X)\big)$};
\node[below of=3, yshift=0.06cm] (8) {$A\oplus B$};
\node[below of=4, yshift=0.06cm] (9) {$A \oplus C$};
\node[below of=5, yshift=0.06cm] (10) {$0.$};
\draw[->,swap](1) to node {\small $b_{n+2}$} (2);
\draw[->](2) to node {} (3);
\draw[->,swap](3) to node {\small $h_{n+1}$} (4);
\draw[->](4) to node {} (5);
\draw[->](6) to node{\small $b_{n+2}$} (7);
\draw[->](7) to node {} (8);
\draw[->](8) to node {\small $h_{n+1}$} (9);
\draw[->](9) to node {} (10);
\draw[->](1) to node {\small $\id$} (6);
\draw[->] (2) to node {\small $\id$} (7);
\draw[->] (3) to node {\small $f_A\oplus \id_B$} (8);
\draw[->] (4) to node {\small $f_A\oplus \id_C$} (9);
\end{tikzpicture}
\end{center}
Hence $(\id, f_A\oplus \id_C,\id)\in\B^{n+2}(X)$, and by Theorem \ref{th:detectingfunctor} there exists
$f\in\E(X)$ such that $H_{n+1}(f) = f_A\oplus\id_C$, $H_{n+2}(f)=\id$ and $H_n(f)=\id$.
\end{proof}

The following lemma is crucial in the proof of Theorems \ref{th:realisabilitygr3} and \ref{th:realisabilityn2}:
\begin{lemma}\label{lemma:splitsummands}
Let $X$ be an $A_n^2$-polyhedron, $n \geq 2 $. Suppose that
there exist $h_{n+1}$-split subgroups of $H_{n+1}(X)$. Then:
\begin{enumerate}
\item If $n\ge 3$, $\B^{n+2}(X)$ is either trivial or it has elements of even order.
\item If $\B^{4}(X)$ is finite and non trivial, then it has elements of even order.
\end{enumerate}
\end{lemma}

\begin{proof}
First of all,  observe that we just need to consider when $H_n(X)$ is an elementary abelian $2$-group. In other case, the result is a consequence of Lemma \ref{lemma:Hnnot2group}.
	
Let $A$ be an arbitrary $h_{n+1}$-split subgroup of $H_{n+1}(X)$.  If $A\ne \Z_2$,  there is an involution $\iota\in\Aut(A)$ that induces,  by Lemma \ref{lemma:autsplitsummands}, an element $(\id, \iota\oplus \id, \id) \in \B^{n+2}(X)$ of order $2$, and the result follows. Hence we can assume that every $h_{n+1}$-split subgroup of $H_{n+1}(X)$ is $\Z_2$.

Both assumptions, $H_n(X)$ being an elementary abelian $2$-group and every $h_{n+1}$-split subgroup of $H_{n+1}(X)$ being $\Z_2$, imply that $H_{n+1}(X)$ is a finite $2$-group. Indeed, since $H_n(X)$ is finitely generated, $\Gamma_n^1\big(H_n(X)\big)$ is a finite $2$-group and so is $\coker b_{n+2}$. Then, since $H_{n+1}(X)$ is also finitely generated, any direct summand of $H_{n+1}(X)$ which is not a $2$-group would be $h_{n+1}$-split, contradicting our assumption that every $h_{n+1}$-split subgroup of $H_{n+1}(X)$ is $\Z_2$.
	
	To prove our lemma, we start with the case $A=H_{n+1}(X)$ is $h_{n+1}$-split.

	When $H_{n+2}(X)=0$, the $\Gamma$-sequence of $X$ becomes then the short exact sequence
	\[0\to \Gamma_n^1\big(H_n(X)\big)\to \Gamma_n^1\big(H_n(X)\big)\oplus \Z_2\to \Z_2\to 0.\]
	Notice that any automorphism of order $2$ in $H_n(X)$ yields an automorphism of order $2$ in $\Gamma_n^1\big(H_n(X)\big)$ since $\Gamma_n^1$ is injective on morphisms: it is immediate for $n\ge 3$, and for $n=2$ we apply Lemma \ref{lemma:gammainjective}. As our sequence is split, any $f\in \Aut\big(H_n(X)\big)$ induces the $\Gamma$-isomorphism $(\id,\id, f)$ of the same order. Hence, for $H_n(X)\ne \Z_2$ it suffices to consider an involution. For $H_n(X)=\Z_2$, since by hypothesis $H_{n+1}(X)=\Z_2$ and $H_{n+2}(X)=0$, the only $\Gamma$-isomorphism is $(\id,\id,\id)$ and therefore $\B^{n+2}(X)$ is trivial as claimed.

When $H_{n+2}(X)\ne 0$, for $n\ge 3$ the result follows directly from Lemma \ref{lemma:homologynot2group}. For $n=2$ we also assume that $\B^4(X)$ is finite and non-trivial. Hence, since $H_2(X)$ is an elementary abelian $2$-group, Proposition \ref{prop:mainfinite} implies that $H_4(X)=\Z$. Then, if a $\Gamma$-isomorphism of the form $(-\id,f,\id)$ exists, it is of even order. In particular, if $\Im b_4$ is a subgroup of $\Gamma\big(H_2(X)\big)$ of order $2$, $(-\id,\id,\id)$ is a $\Gamma$-isomorphism of even order.

Assume otherwise that $\Im b_4$ is a group of order $4$. If a $\Gamma$-isomorphism $(\id,f,\id)$ of odd order exists, then $\Gamma(f)\circ b_4 = b_4$. In this situation, by Lemma \ref{lemma:h3splitIfOddThenEven} for $\chi = b_4(1)$, there exists $g\in\Aut\big(H_2(X)\big)$ an automorphism of order $2$ such that $\Gamma(g)  b_4 (1)= b_4 (1)$. Moreover, as we are in the case $A = H_3(X)$ being $h_3$-split, $(\id,g,\id)\in\mathcal{B}^4(X)$ is a $\Gamma$-isomorphism of order $2$.

We deal now with the case $A\lvertneqq H_{n+1}(X)$. Since $A=\Z_2$ is a proper $h_{n+1}$-split subgroup of $H_{n+1}(X)$, there exist non-trivial groups $B$ and $C$ such that
 $$\begin{array}{rcl}
\pi_{n+1}(X)=\Z_2\oplus B & \xrightarrow{h_{n+1}} & \Z_2\oplus C=H_{n+1}(X)\\
(t,b) & \longmapsto & (t,g(b))
\end{array}$$
for some group morphism $B\xrightarrow{g} C$. Moreover,
$H_{n+1}(X)$ is a finite $2$-group, thus $C$ is a (non-trivial) finite $2$-group and there exists an epimorphism $C\xrightarrow{\tau} \Z_2$.

Define $f\in\Aut(\Z_2\oplus C)=\Aut\big(H_{n+1}(X)\big)$, and $\Omega\in\Aut\big(\Z_2\oplus B\big)=\Aut\big(\pi_{n+1}(X)\big)$ to be the non-trivial involutions given by $f(t,c)=\big(t+\tau(c),c\big)$ and $\Omega(t,b)=\big(t+\tau(g(b)), b\big)$. By construction, $h_{n+1}\Omega=f h_{n+1}$, and if $(t,b)\in\coker b_{n+2}=\ker h_{n+1}$ (thus $g(b)=0$), then $\Omega(t,b)=(t,b)$. In other words, $(\id, f,\id)\in\B^{n+2}(X)$ and it has order $2$.
\end{proof}

We now prove our main results.

\begin{proof}[Proof of Theorem \ref{th:realisabilitygr3}] Assume that $H_n(X)$ and $H_{n+1}(X)$ are elementary abelian $2$-groups, and $H_{n+2}(X)=0$. Otherwise, there would already be elements of order $2$ in $\B^{n+2}(X)$ as a consequence of Lemma \ref{lemma:homologynot2group}.

	Write $H_n(X)=\oplus_I\Z_2$, $I$ an ordered set. Since $n\ge 3$, $\Gamma_n^1 = -\otimes \Z_2 $, so $\Gamma_n^1\big(H_n(X)\big)=H_n(X)$. We can also assume that there are no subgroups in $H_{n+1}(X)$ that are $h_{n+1}$-split. In other case, we would deduce from Lemma \ref{lemma:splitsummands} that there are elements of order $2$ in $\B^{n+2}(X)$. Thus $H_{n+1}(X)=\oplus_J \Z_2$ with $J\subset I$, and the $\Gamma$-sequence corresponding to $X$ is
	\[
	\textstyle{0\to \bigoplus_I \Z_2 \xrightarrow{b} \left(\bigoplus_{I-J} \Z_2 \right)\oplus \left(\bigoplus_J \Z_4\right)	\xrightarrow{h} \bigoplus_J \Z_2\to 0}.
	\]
	We may rewrite the sequence as
		\[
	\textstyle{0\to \left(\bigoplus_{I-J} \Z_2\right) \oplus \left(\bigoplus_J \Z_2\right)\xrightarrow{b} \left(\bigoplus_{I-J} \Z_2 \right)\oplus \left(\bigoplus_J \Z_4\right)	\xrightarrow{h} \bigoplus_J \Z_2\to 0}
	\]
	and assume that $b(x, y)=(x,2y)$ and $h(x, y)=y\hspace{-4pt}\mod 2$. It is clear that any $f\in\Aut\big(\bigoplus_{I-J} \Z_2\big)$ induces a $\Gamma$-isomorphism $(0,\id, f\oplus\id)$ of the same order.
	
	On the one hand, for $|I-J|\ge 2$, $\bigoplus_{I-J} \Z_2$ has an involution and therefore $\B^{n+2}(X)$ has elements of even order.  On the other hand, for $|I-J|<2$, we consider the remaining possibilities.
	
	Suppose that $|I-J|=1$. Then, $\pi_{n+1}(X)=\Z_2\oplus(\oplus_J\Z_4)$. If $J$ is trivial, $\B^{n+2}(X)$ is clearly trivial as well. Otherwise, suppose that $I-J=\{i\}$ and choose $j\in J$. Define $f\in \Aut \big(\Z_2\oplus\Z_2\oplus (\oplus_{I-\{i, j\}}\Z_2)\big)$ by $f(x, y, z)=(x, x+y, z)$ and $g\in \Aut \big(\Z_2\oplus\Z_4\oplus (\oplus_{I-\{i, j\}}\Z_4)\big)$ by $g(x, y, z)=(x,2x+y, z)$. Then $(\id,\id, f)$ is a $\Gamma$-isomorphism of order $2$ since we have a commutative diagram
	\begin{center}
\begin{tikzpicture}
\tikzset{node distance=0.148\textwidth, auto}
\node(1) {0};
\node[right of=1, xshift=0.6cm] (2) {\small $\Z_2\oplus \Z_2\oplus(\oplus_{I-\{i, j\}} \Z_2)$};
\node[right of=2, xshift=2.3cm] (3) {\small $\Z_2\oplus \Z_4\oplus(\oplus_{I-\{i, j\}} \Z_4)$};
\node[right of=3, xshift=1.8cm] (4) {\small $\Z_2 \oplus(\oplus_{J-\{j\}} \Z_2)$};
\node[right of=4] (5) {$0$};
\node[below of=1, yshift=0.05cm] (6) {\small $0$};
\node[below of=2, yshift=0.05cm] (7) {\small $\Z_2\oplus \Z_2\oplus(\oplus_{I-\{i, j\}} \Z_2)$};
\node[below of=3, yshift=0.05cm] (8) {\small $\Z_2\oplus \Z_4\oplus(\oplus_{I-\{i, j\}} \Z_4)$};
\node[below of=4, yshift=0.05cm] (9) {\small $\Z_2 \oplus(\oplus_{J-\{j\}} \Z_2)$};
\node[below of=5, yshift=0.05cm] (10) {\small $0.$};
\draw[->](1) to node {} (2);
\draw[->](2) to node {} (3);
\draw[->](3) to node {} (4);
\draw[->](4) to node {} (5);
\draw[->](6) to node{} (7);
\draw[->](7) to node {} (8);
\draw[->](8) to node {} (9);
\draw[->](9) to node {} (10);
\draw[->] (2) to node {\small $f$} (7);
\draw[->] (3) to node {\small $g$} (8);
\draw[->] (4) to node {\small $\id$} (9);
\end{tikzpicture}
\end{center}

	Suppose that  $I=J$. If $H_n(X)=H_{n+1}(X)=\Z_2$, $\B^{n+2}(X)$ is trivial. If not, choose $i, j\in I$ and define maps $f\in\Aut\big(\Z_2\oplus\Z_2\oplus(\oplus_{I-\{i, j\}}\Z_2)\big)$  by $f(x, y, z)=(y, x, z)$, and $g\in\Aut\big(\Z_4\oplus\Z_4\oplus(\oplus_{I-\{i, j\}}\Z_4)\big)$  by $g(x, y, z)=(y, x, z)$. We have the following commutative diagram
	\begin{center}
\begin{tikzpicture}
\tikzset{node distance=0.148\textwidth, auto}
\node(1) {\small $0$};
\node[right of=1, xshift=0.3cm] (2) {\small $\Z_2\oplus \Z_2\oplus(\oplus_{I-\{i, j\}} \Z_2)$};
\node[right of=2, xshift=2cm] (3) {\small $\Z_4\oplus \Z_4\oplus(\oplus_{I-\{i, j\}} \Z_4)$};
\node[right of=3, xshift=2cm] (4) {\small $\Z_2\oplus\Z_2 \oplus(\oplus_{I-\{i, j\}} \Z_2 )$};
\node[right of=4, xshift=0.4cm] (5) {\small $0$};
\node[below of=1, yshift=0.05cm] (6) {\small $0$};
\node[below of=2, yshift=0.05cm] (7) {\small $\Z_2\oplus \Z_2\oplus(\oplus_{I-\{i, j\}} \Z_2)$};
\node[below of=3, yshift=0.05cm] (8) {\small $\Z_4\oplus \Z_4\oplus(\oplus_{I-\{i, j\}} \Z_4)$};
\node[below of=4, yshift=0.05cm] (9) {\small $\Z_2\oplus\Z_2 \oplus(\oplus_{I-\{i, j\}} \Z_2) $};
\node[below of=5, yshift=0.05cm] (10) {\small $0.$};
\draw[->](1) to node {} (2);
\draw[->](2) to node {} (3);
\draw[->](3) to node {} (4);
\draw[->](4) to node {} (5);
\draw[->](6) to node{} (7);
\draw[->](7) to node {} (8);
\draw[->](8) to node {} (9);
\draw[->](9) to node {} (10);
\draw[->] (2) to node {\small $f$} (7);
\draw[->] (3) to node {\small $g$} (8);
\draw[->] (4) to node {\small $f$} (9);
\end{tikzpicture}
\end{center}
 Then, $(0, f, f)$ is a $\Gamma$-isomorphism of order $2$.
\end{proof}

As a consequence, we obtain a negative answer to the problem of realising groups as self-homotopy equivalences of $A_n^2$-polyhedra:
\begin{corollary}\label{cor:KhanReal}
	Let $G$ be a non nilpotent finite group of odd order. Then, for any $n\ge 3$ and for any $A_n^2$-polyhedron $X$, $G \not\cong \E(X)$.
\end{corollary}

\begin{proof}
	Assume that there exists an $A_n^2$-polyhedron $X$ such that $\E(X)\cong G$. Then, if $\E(X)\ne \E_*(X)$, the quotient $\E(X)/\E_*(X)$ is a finite group of odd order, which contradicts Theorem \ref{th:realisabilitygr3}. Thus $G \cong \E(X) = \E_*(X)$. However, since $X$ is a $1$-connected and finite-dimensional $CW$-complex, $\E_*(X)$ is a nilpotent group, \cite[Theorem D]{DroZab79}, which contradicts the fact that $G$ is non nilpotent.
\end{proof}

We end this paper by proving our second main result:
\begin{proof}[Proof of Theorem \ref{th:realisabilityn2}]
	By hypothesis $\B^4(X)$ is a finite group of odd order. From Lemma \ref{lemma:Hnnot2group} we deduce that $H_2(X)$ is an elementary abelian $2$-group and from Proposition \ref{prop:propertiesgammafunctor} that $\Gamma\big(H_2(X)\big)$ is a $2$-group. In particular, every element of $\Gamma\big(H_2(X)\big)$ is of finite order, and therefore, by Proposition \ref{prop:mainfinite}, $\rank H_4(X) \leq 1$ so we have Theorem \ref{th:realisabilityn2}.(1). Now, any element in $\B^4(X)$  is of the form $(0, f_2, f_3)$  if $H_4(X) = 0$ or $(\id, f_2, f_3)$ if $H_4 (X) = \Z$. Notice that a $\Gamma$-morphism of the form $(-\id, f_2, f_3)$ has even order thus it cannot be a $\Gamma$-isomorphism under our hypothesis. Therefore, if $H_4 (X) = \Z$, then $b_4(1)$ generates a $\Z_4$ factor in $\Gamma\big(H_2(X)\big)$, and under our hypothesis the equation $$\rank \Gamma\big(H_2(X)\big)=\rank H_4(X) + \rank(\coker b_4)$$ holds for $\rank H_4(X) \leq 1$.
	
	Observe that any $\Gamma$-isomorphism of $X$ induces a chain morphism of the short exact sequence
	\[
	0\to \coker b_4 \to \pi_3(X) \xrightarrow{h_3} H_3(X)\to 0.
	\]
	We will draw our conclusions from this induced morphism, which can be seen as an automorphism of $\pi_3(X)$ that maps the subgroup $i_2(\coker b_4)$ to itself, thus inducing an isomorphism on the quotient, $H_3(X)$.
		
	As we mentioned above, $\Gamma\big(H_2(X)\big)$ is a $2$-group. Then $\coker b_4$ is a quotient of a $2$-group so a $2$-group itself. We claim that $H_3(X)$ is also a $2$-group: otherwise, $H_3(X)$ has a summand whose order is either infinite or odd and therefore this summand would be $h_3$-split, which from Lemma \ref{lemma:splitsummands} implies that $\B^4(X)$ has elements of even order, leading to a contradiction. Since $\coker b_4$ and $H_3(X)$ are $2$-groups, so is $\pi_3(X)$, proving thus Theorem \ref{th:realisabilityn2}.(2).
	
	Moreover, no subgroup of $H_3(X)$ can be $h_3$-split as a consequence of Lemma \ref{lemma:splitsummands}, and thus, $\rank H_3(X) \le \rank (\coker b_4) = \rank \Gamma\big(H_2(X)\big) - \rank H_4(X)$. We can compute $\rank \Gamma\big(H_2(X)\big)$ using Proposition \ref{prop:propertiesgammafunctor} and immediately obtain Theorem \ref{th:realisabilityn2}.(3).
		
	Now for a $2$-group $G$, define the subgroup $\Omega_1(G)=\{g\in G : \operatorname{ord}(g)|2\}$.  One can easily check that $\Omega_1\big(\pi_3(X)\big)\le i_2(\coker b_4)$ and, from \cite[Ch. 5, Theorem 2.4]{Gor68}, we obtain that any automorphism of odd order of $\pi_3(X)$ acting as the identity on $i_2(\coker b_4)$ must be the identity.

	Then, if $(\id, f_3, f_2)\in \B^4(X)$ is a $\Gamma$-morphism with $f_3$ non-trivial, $f_3$ has odd order, so we may assume that $\Omega\colon \pi_3(X)\to\pi_3(X)$ (see Definition \ref{def:gammasequences}) has odd order too. By the argument above, it must induce a non-trivial homomorphism on $i_2(\coker b_4)$ and therefore $f_2$ is non-trivial as well. Thus, the natural action of $\B^4(X)$ on $H_2(X)$ must be faithful, since any $\Gamma$-automorphism $(\id, f_3, f_2)\in \B^4(X)$ induces a non-trivial $f_2\in \Aut\big(H_2(X)\big)$. Then, Theorem \ref{th:realisabilityn2}.(4) follows.
\end{proof}

\providecommand{\bysame}{\leavevmode\hbox to3em{\hrulefill}\thinspace}
\providecommand{\MR}{\relax\ifhmode\unskip\space\fi MR }
\providecommand{\MRhref}[2]{%
  \href{http://www.ams.org/mathscinet-getitem?mr=#1}{#2}
}
\providecommand{\href}[2]{#2}

\end{document}